\documentclass[english]{amsart}   
\usepackage[T1]{fontenc}
\usepackage[latin1]{inputenc}

\usepackage[english]{babel}
\usepackage{amsmath,amsfonts,amssymb,amsthm}
\usepackage{mathrsfs}
\usepackage{graphicx}
\usepackage[all]{xy}
\usepackage{comment}
\usepackage{bbold}

\usepackage{bbold}

	\advance \topmargin by -\headheight
	\advance \topmargin by -\headsep

	\evensidemargin \oddsidemargin
	\def\crt{^{\scriptscriptstyle {\it CRT}}}
	\def\scrt{_{\scriptscriptstyle {\it CRT}}}
	\def\ct{{\it CRT}}

	\newcommand{\ve}{\varepsilon}

	\newcommand{\Hil}{\mathbb H}
	\newcommand{\F}{\mathbb F}

	\def\sur{_\R}
	\def\suc{_\C}

	\def\so{_{\scriptscriptstyle O}}
	\def\su{_{\scriptscriptstyle U}}
	\def\st{_{\scriptscriptstyle T}}

	\newcommand{\id}{\ensuremath{\operatorname{id}}}

	\newcommand{\ftn}[3]{ #1 : #2 \rightarrow #3 }
	
	\newcommand{\mc}[1]{\mathcal{#1}}

	\newcommand{\sm}[4]{       \left( \begin{smallmatrix} 
					{#1} & {#2} \\ {#3} & {#4}
                   		    \end{smallmatrix} \right)         }

	\newcommand{\catr}{\textbf{C*$\R$-Alg}}

	\newcommand{\catabgp}{\textbf{Ab}}
	\newcommand{\catkk}{\textbf{KK}}

	\newcommand{\multialg}[1]{\mathcal{M}(#1)}
	
	\newcommand{\norm}[1]{\left\| #1 \right\|}
	\newcommand{\im}{\operatorname{Im}}
	\newcommand{\asy}[1]{\left\langle #1 \right\rangle}

	\newcommand{\Z}{\mathbb{Z}}
	\newcommand{\Ham}{\mathbb{H}}
	\newcommand{\C}{\mathbb{C}}
	
	\newcommand{\R}{\mathbb{R}}
	\newcommand{\N}{\mathbb{N}}
	\newcommand{\K}{\mathcal{K}}
	\newcommand{\kk}{\operatorname{KK}}
	
	\newcommand{\ext}{\operatorname{Ext}}
	\newcommand{\hoom}{\operatorname{Hom}}

	\theoremstyle{plain}
	\newtheorem{thm}{Theorem}[section]
	\newtheorem{lemma}[thm]{Lemma}
	\newtheorem{theor}[thm]{Theorem}
	\newtheorem{propo}[thm]{Proposition}
	\newtheorem{corol}[thm]{Corollary}

	\theoremstyle{definition}
	\newtheorem{defin}[thm]{Definition}

	\newtheorem{probl}[thm]{Problem}

	\numberwithin{equation}{section}
	\numberwithin{figure}{section}

\begin{document}

	\title[Pictures of $\kk$-theory for real $C^{*}$-algebras ]{Pictures of $\kk$-theory for real $C^{*}$-algebras  \\
		and almost commuting matrices}

	\author{Jeffrey L. Boersema}	
	\address{Department of Mathematics \\
	Seattle University \\
	Seattle, WA 98133 USA}
	\email{boersema@seattle.edu}

	\author{Terry A. Loring}
	\address{University of New Mexico \\ Department of Mathematics and Statistics \\
	 Albuquerque, New Mexico\\
	  87131, USA}
	  \email{loring@math.unm.edu}
	  
	\author{Efren Ruiz}
        \address{Department of Mathematics\\University of Hawaii,
Hilo\\200 W. Kawili St.\\
Hilo, Hawaii\\
96720-4091 USA}
        \email{ruize@hawaii.edu}
        \date{\today}

	\begin{abstract}
We give a systematic account of the various pictures of $\kk$-theory for real $C^*$-algebras, proving natural isomorphisms between the groups that arise from each picture.  As part of this project, we develop the universal properties of $\kk$-theory, and we use \ct-structures to prove that a natural transformation $F(A) \rightarrow G(A)$ between homotopy equivalent, stable, half-exact functors defined on real $C^*$-algebras is an isomorphism provided it is an isomorphism on the smaller class of $C^*$-algebras. Finally, we develop $E$-theory for real $C^*$-algebras and use that to obtain new negative results regarding the problem of approximating almost commuting real matrices by exactly commuting real matrices.
	\end{abstract}

        \maketitle

\section{Introduction and Preliminaries}

A real $C^*$-algebra is a Banach $*$-algebra $A$ over the real numbers such that $\| x^* x \| = \|x\|^2$ holds for all $x$ and such that every element of the form $1 + x^* x$ is invertible in the unitization of $A$ (see \cite{schroderbook}). In this paper, we will adopt the term $R^*$-{\it algebra} instead. As is well-known, every $R^*$-algebra is isometrically isomorphic to a closed $*$-algebra of bounded operators on a Hilbert space over $\R$. In addition, every $R^*$-algebra is isomorphic to the $*$-algebra of fixed elements of a $C^*$-algebra with a conjugate linear involution.

In Kasparov's seminal paper \cite{kasparov80} introducing $\kk$-theory, he simultaneously considered both $R^*$-algebras and $C^*$-algebras. Since then, many alternate but equivalent or closely related pictures of $\kk$-theory have been introduced and developed by various authors 
(\cite{ConnesHigson90}, \cite{cuntz87}, \cite{higson87}, \cite{higson90}, \cite{ManuilovThomsen04}, \cite{thomsen90}). The ability to move among the various pictures has contributed immensely to the utility of $\kk$-theory as a tool for solving problems. However, these authors have not consistently followed Kasparov's lead in considering the real case along with the complex.

In recent years, substantial progress has been made in developing the tools to study $R^{*}$-algebras including the development of united $K$-theory and the universal coefficient theorem in \cite{boersema02} and \cite{boersema04}. This has led to a classification of purely infinite simple $R^*$-algebras (in \cite{BRS}) and the classification of real forms of UHF-algebras that are stable over the CAR-algebra (in \cite{stacey}).

Given the centrality of $\kk$-theory for these projects, there has been a need to develop a systematic account of the various pictures of $\kk$-theory for $R^*$-algebras. In this paper, we will develop several of the alternate pictures of $\kk$-theory in the context of $R^*$-algebras and prove the appropriate equivalent theorems. In particular, in this paper we will consider the following pictures of $\kk$-theory and prove appropriate equivalence theorems for each:  the standard Kasparov bimodule picture of $\kk$-theory, the Fredholm picture (both in Section~2), the universal property picture (Section~3), and suspended $E$-theory using asymptotic morphisms (Section~4). Since this aspect of the theory goes through in the real case very much like the complex case, we present these results in a survey-like overview.

As part of this, we prove a more general theorem which reduces the work required to replicate many of these equivalent theorems and 
promises to ease the way for similar projects in the future. Suppose that
$\mu \colon F \rightarrow G$
is a natural transformation between homotopy invariant, stable, half exact functors. We prove that if $\mu$ is an isomorphism for all $C^*$-algebras, then it is an isomorphism for $R^*$-algebras. This is accomplished in Section~3 when we develop the universal properties of $\kk$-theory and $K$-theory for $R^*$-algebras.

In the last two sections, we will apply these ideas, using $\kk$-theory to prove the existence of certain asymptotic morphisms, which in turn is used to obtain new results for the problem of approximating a set of almost commuting matrices over the field of real numbers. In particular, let the Halmos number be the largest integer $d$ such that whenever $d$ real self-adjoint matrices almost commute (pairwise) they can be approximated by $d$ pairwise commuting matrices.  More precisely, for all $\ve > 0$ there should be a $\delta > 0$ such that if $\{ H_i \}_{i = 1}^d$ is a collection of $d$ self-adjoint matrices such that 
$$ \| H_r \| \leq 1 \text{ \qquad and \qquad} \| [ H_r, H_s ] \| \leq \delta \; ,$$
for all $r,s$, then there exists a collection $\{ K_i \}_{i = 1}^d$ of self-adjoint matrices such that
$$ \| K_r \| \leq 1 \text{ \qquad and \qquad} \| [ K_r, K_s ] \| = 0  \text{ \qquad and \qquad} \| H_r - K_r \| \leq \ve   \; .$$
Furthermore, the dependence of $\delta$ on $\ve$ must be uniform, independent of the dimension of the matrices $H_r$. It is shown in 
\cite{LorSorensenTorus} that in the context of real matrices, the statement is true for $d = 2$.  We will show in Section 7 the statement is false for $d = 5$.  Therefore, the Halmos number for real matrices is between 2 and 4, inclusive. 

\section{The Standard and Fredholm Pictures of $\kk$-Theory}

We start with the following definition of $\kk$-theory. It is essentially the same as that in \cite{kasparov80} where it was simultaneously developed for both $R^{*}$-algebras and $C^{*}$-algebras. It also appears in Section~2.3 of \cite{schroderbook} for $R^{*}$-algebras, Section~2.1 of \cite{jensenthomsenbook}, and the Appendix of \cite{higson87}.

\begin{defin}\label{def:kbimod}
Let $A$ and $B$ be graded separable $R^{*}$-algebras with $B$ assumed to be $\sigma$-unital.
\begin{itemize}
\item[(i)]  A Kasparov ($A$-$B$)-bimodule is a triple $( E , \phi , T )$ where $E$ is a countably generated graded real Hilbert $B$-module, $\ftn{ \phi }{ A }{ \mc{L}\sur ( E ) }$ is a graded $*$-homomorphism, and $T$ is an element of $\mc{L}\sur ( E )$ of degree 1 such that
\begin{equation*}
( T - T^{*} ) \phi ( a ), ( T^{2} - 1 ) \phi ( a ) , \ \mathrm{and} \ [ T , \phi ( a ) ] 
\end{equation*}
lie in $\K\sur(E)$ for all $a \in A$.

\item[(ii)] Two triples $( E_{i} , \phi_{i} , T_{i} )$ are \emph{unitarily equivalent} if there is unitary $U$ in $\mc{L}\sur ( E_{0} , E_{1} )$, of degree zero, intertwining both $\phi_{i}$ and $T_{i}$ in the appropriate way.

\item[(iii)] Let $(E, \phi, T)$ be a Kasparov ($A$-$B$)-bimodule and let $\beta \colon B \rightarrow B'$ be a $*$-homomorphism of $R^{*}$-algebras.  Then the pushed-forward Kasparov ($A$-$B'$)-bimodule is defined by $$\beta_*(E,\phi,T) = (E \hat{\otimes}_\beta B', \phi \hat{\otimes} 1, T \hat{\otimes} 1) \; .$$

\item[(iv)]  Two Kasparov ($A$-$B$)-bimodules $(E_{i} , \phi_{i} , T_{i} )$ for $i = 0,1$ are \emph{homotopic} if there is a Kasparov bimodule ($A$-$IB)$, say $(E, \phi , T )$, such that $(\ve_i)_* (E, \phi, T)$ and $(E_i, \phi_i, T_i)$ are unitarily equivalent for $i = 0,1$, where $IB = C ([ 0 ,1] , B )$ and $\ve_{i}$ denotes the evaluation map.

\item[(v)]  A triple $(E, \phi , T )$ is \emph{degenerate} if the elements \begin{equation*}
( T - T^{*} ) \phi ( a ), ( T^{2} - 1 ) \phi ( a ) , \ \mathrm{and} \ [ T , \phi ( a ) ]
\end{equation*}
are zero for all $a \in A$.  By Proposition~2.3.3 of \cite{schroderbook}, degenerate bimodules are homotopic to trivial bimodules.
    
\item[(vi)] $\kk( A , B )$ is defined to be the set of homotopy equivalence classes of Kasparov ($A$-$B$)-bimodules.  
\end{itemize}
\end{defin}     

The following theorem summarizes the principal properties of $\kk$-theory for $R^{*}$-algebras from Chapter~2 of \cite{schroderbook}.

\begin{propo} \label{kkproperties}
$\kk(A,B)$ is an abelian group for separable $A$ and $\sigma$-unital $B$.
As a functor on separable $R^{*}$-algebras
(contravariant in the first argument and covariant in the second argument), it is homotopy invariant, stable, and has split exact sequences in both arguments.  Furthermore, there is a natural associate pairing (the intersection product)
$$\otimes_{KK} \colon \kk(A, C \otimes B) \otimes \kk(C \otimes A', B') \rightarrow \kk(A \otimes A', B \otimes B')  \; .$$
\end{propo}

We now turn to the Fredholm picture of $\kk$-theory, which was developed in \cite{higson87} in the context of $C^{*}$-algebras. A simplified picture along these lines also appears in Chapter 4 of the text \cite{jensenthomsenbook}. As we show, the approach goes through the same for $R^{*}$-algebras, as follows.

\begin{defin}
Let $A$ and $B$ be separable $R^{*}$-algebras.
\begin{itemize}
\item[(i)]  A triple $( \phi_{+} , \phi_{-} , U )$, where 
$\ftn{ \phi_{ \pm } }{ A }{ \mc{M}({ \K_{\R} \otimes B }) }$ are $*$-homomorphisms, and $U$ is an element of $\mc{M}({ \K_{\R} \otimes B })$ such that
\begin{equation*}
U \phi_{+} ( a ) - \phi_{-} ( a ) U, \ \phi_{+} ( a ) ( U^{*} U - 1 ), \ \mathrm{and} \ \phi_{-} ( a ) ( U U^{*} - 1 ) 
\end{equation*}
lie in $\K_{\R} \otimes B$ for all $a \in A$ is called a $\kk(A,B)$-cycle.  

\item[(ii)]  Two $\kk ( A , B )$-cycles $( \phi_{+}^{1} , \phi_{-}^{1} , U^{1} )$ and $( \phi_{+}^{2} , \phi_{-}^{2} , U^{2} )$ are \emph{homotopic} if there is a $\kk ( A , IB  )$-cycle $(\phi_{+} , \phi_{-} , U )$ such that 
$$( \ve_{i} \phi_{+}, \ve_{i} \phi_{-} , \ve_{i} ( U ) ) = ( \phi_{+}^{i} , \phi_{-}^{i} , U^{i} ),$$
where 
$\ftn{ \ve_{i} } { \mc{M}( \K_{\R} \otimes IB) }{ \mc{M}( \K_{\R} \otimes B ) }$ is induced by evaluation at $i$.

\item[(iii)]  A $\kk ( A , B )$-cycle $(\psi_{+} , \psi_{-} , V )$ is \emph{degenerate} if the elements 
\begin{equation*}
V \psi_{+} ( a ) - \psi_{-} ( a ) V, \ \psi_{+} ( a ) ( V^{*} V - 1 ), \ \mathrm{and} \ \psi_{-} ( a ) ( V V^{*} - 1 ) 
\end{equation*}
are zero for all $a \in A$.     

\item[(iv)]  The sum $( \phi_{+} , \phi_{-} , U ) \oplus ( \psi_{+} , \psi_{-} ,V )$ of two $\kk ( A , B )$-cycles is the $\kk ( A , B )$-cycle 
\begin{equation*}
\left( \left( 
\begin{matrix}
\phi_{+} & 0 \\
0 & \psi_{+}
\end{matrix}
\right), 
\left( 
\begin{matrix}
\phi_{-}  & 0 \\
0 & \psi_{-}  
\end{matrix}
\right),
\left( 
\begin{matrix}
U & 0 \\
0 & V
\end{matrix}
\right)
\right)
\end{equation*}
where the algebra $M_2 ( \mc{M}( \K_{\R} \otimes B ) )$ is identified with $\mc{M}( \K_{\R} \otimes B )$ by means of some $*$-isomorphism $M_2  ( \K_{\R} ) \cong \K_{\R}$, which is unique up to homotopy by Section~1.17 of \cite{kasparov80}.

\item[(v)]  Two cycles $( \phi_{+}^{0}, \phi_{-}^{0} , U^{0} )$ and $( \phi_{+}^{1}, \phi_{-}^{1} , U^{1} )$ are said to be \emph{equivalent} if there exist degenerate cycles $( \psi_{+}^{0} , \psi_{-}^{0} , V^{0} )$ and $( \psi_{+}^{1} , \psi_{-}^{1} , V^{1} )$ such that 
\begin{equation*}
( \phi_{+}^{0} , \phi_{-}^{0} , U^{0} ) \oplus ( \psi_{+}^{0}, \psi_{-}^{0}, V^{0} ) \quad \mathrm{and} \quad ( \phi_{+}^{1} , \phi_{-}^{1} , U^{1} ) \oplus ( \psi_{+}^{1}, \psi_{-}^{1}, V^{1} )
\end{equation*}  
are homotopic.  

\item[(vi)] $\textbf{KK} ( A , B )$ is defined to be the set of equivalence classes of $\kk ( A , B )$-cycles.

\end{itemize}
\end{defin}

The following lemma is the real version of Lemma~2.3 of \cite{higson87}, with the same proof.

\begin{propo}  \label{lem:fhgrp}
$\textbf{KK}( A , B )$ is an abelian group, for separable $R^{*}$-algebras $A$ and $B$.  As a functor it is contravariant in the first argument and covariant in the second argument.
\end{propo}

\begin{theor}\label{prop:eqdef}
Let $A$ and $B$ be separable $R^{*}$-algebras (with the trivial grading).  Then $\kk ( A , B )$ is isomorphic to $\textbf{KK} ( A , B )$.
\end{theor}        

\begin{proof}[Sketch of proof.]
Let $\Hil_B$ be the Hilbert $B$-module consisting of all sequences $\{b_n\}_{n = 1}^\infty$ in $B$ such that $\sum_{n = 1}^\infty b_n^* b_n$ converges. Let $\hat{\Hil}_B = \Hil_B \oplus \Hil_B$ be the graded Hilbert $B$-module with $\hat{\Hil}_B^{(0)} = \Hil_B \oplus 0$ and $\hat{\Hil}_B^{(1)} = 0 \oplus \Hil_B$. This induces a grading on $\mathcal{L}\sur(\hat{ \Hil}_B) \cong  \multialg{ \K_{\R} \otimes B } $.

Given a $\kk ( A , B )$-cycle $x = ( \phi_{+} , \phi_{-} , U )$ we define 
\begin{equation*}
\alpha ( x ) = \left( \hat{ \Hil }_{ B } , \left( \begin{matrix} \phi_{+} & 0 \\ 0 & \phi_{-} \end{matrix} \right)  ,  \left( \begin{matrix} 0 & U^{*} \\  U & 0 \end{matrix} \right) \right) \;.
\end{equation*}
It is verified that $ \left( \begin{matrix} 0 & U^{*} \\  U & 0 \end{matrix} \right) $ has degree $1$ and that $\alpha(x)$ is indeed a Kasparov ($A$-$B$) bimodule. That $\alpha$ induces a well-defined isomorphism
$$\overline{\alpha} \colon  \textbf{KK} ( A , B ) \rightarrow \kk(A, B) \; $$
can be shown by adapting the methods of the Appendix of \cite{higson87} or Chapter 4 of \cite{jensenthomsenbook}.
\end{proof}
        
\section{The Universal Property of $\kk$-Theory } 

Let $F$ be a functor from the category {\catr}  of separable $R^{*}$-algebras to the cateogry {\catabgp} of abelian groups.  We say that $F$ is
\begin{enumerate}
\item[(i)] {\it homotopy invariant} if $(\alpha_1)_* = (\alpha_2)_*$ whenever $\alpha_1$ and $\alpha_2$ are homotopic $*$-homomorphisms on the level of $R^{*}$-algebras.
\item[(ii)] {\it stable} if $(e_A)_* \colon F(A) \rightarrow F(\K_{\R} \otimes A)$ is an isomorphism for the inclusion $e_A \colon A \hookrightarrow \K_{\R} \otimes A$ defined via any rank one projection.
\item[(iii)] {\it split exact} if any split exact sequence of separable $R^{*}$-algebras
$$0 \rightarrow A \rightarrow B \rightarrow C \rightarrow 0$$
induces a split exact sequence 
$$0 \rightarrow F(A) \rightarrow F(B) \rightarrow F(C) \rightarrow 0 \; .$$
\item[(iv)] {\it half exact} if any short exact sequence of separable $R^{*}$-algebras
$$0 \rightarrow A \rightarrow B \rightarrow C \rightarrow 0$$
induces an exact sequence 
$$F(A) \rightarrow F(B) \rightarrow F(C) \; .$$
\end{enumerate}

In what follows we will see that if $F$ is homotopy invariant and half exact, then it is split exact.

\begin{propo} \label{stability}
If $F$ is a functor from {\catr} to {\catabgp} that is homotopy invariant, then the functor $F_s$ defined by 
$F_s(A) = F(\K_{\R} \otimes A)$ is homotopy invariant and stable.
\end{propo}

\begin{proof}
Just as in the complex case (Theorem~4.1.13 of \cite{jensenthomsenbook}), the map $e_{ \K_{\R}} \colon  \K_{\R} \rightarrow \K_{\R} \otimes \K_{\R}$ is homotopic to an isomorphism.
\end{proof}

The following theorem is the version for $R^{*}$-algebras of Theorem~3.7 of \cite{higson87} and Theorem~22.3.1 of \cite{blackadarbook}.

\begin{theor} \label{thm:univprop1}
Let $F$ be a functor from {\catr} to {\catabgp} that is homotopy invariant, stable, and split exact.  Then there is a unique natural pairing $\alpha \colon F(A) \otimes \kk(A,B) \rightarrow F(B)$ such that $\alpha(x \otimes 1_A) = x$ for all $x \in F(A)$ and where $1_A \in \kk(A,A)$ is the class represented by the identity $*$-homomorphism.

Furthermore, the pairing respects the intersection product on $\kk$-theory in the sense that
$$\alpha(\alpha(x \otimes y) \otimes z) = \alpha(x \otimes ( y \otimes_B z)) \colon F(A) \otimes \kk(A,B) \otimes \kk(B,C) \rightarrow F(C) \; .$$
\end{theor}

\begin{proof}
Let $\Phi \in \kk(A,B)$.  Using Theorem~\ref{prop:eqdef} we represent $\Phi$ with a $\kk(A,B)$ cycle and as in Lemma~3.6 of \cite{higson87}, we may assume that this cycle has the form $(\phi_+, \phi_-, 1)$.  We use the same construction as in Definitions~3.3 and 3.4 in \cite{higson87}.  In that setting $F$ is assumed to be a functor from separable $C^{*}$-algebras, but it goes through the same for functors from separable $R^{*}$-algebras to any abelian category.  This construction produces a homomorphism $\Phi_* \colon F(A) \rightarrow F(B)$ and we then define $\alpha(x \otimes \Phi) = \Phi_*(x)$.  The proof of Theorems~3.7 and 3.5 of \cite{higson87} carry over in the real case to show that $\alpha$ is natural, is well-defined, satisfies $\alpha(x \otimes 1_A) = x$, and is unique.

That $\alpha$ respects the Kasparov product follows from the uniqueness statement.
\end{proof}

We also note the contravariant version of the result above.  If $F$ is a contravriant functor, otherwise satisfying the above hypotheses, then there is a pairing $\alpha \colon \kk(A, B) \otimes F(B) \rightarrow F(A)$ such that $\alpha(1_A \otimes x) = x$ for all $x \in F(A)$.

For any $R^{*}$-algebra $A$, we define $SA = \{ f \in C([0,1], A) \mid f(0) = f(1) = 0\}$, or equivalently up to $*$-isomorphism, 
$SA = C_0(\R, A)$. We similarly define $S^{-1}A = \{ f \in C_0(\R, A\suc) \mid f(-x) = \overline{f(x)} \}$. By iteration, $S^n A$ is defined for all $n \in \Z$.  Since $S S^{-1} \R$ is $\kk$-equivalent to $\R$, the formula $S^{n} S^m A \equiv S^{n+m} A$ holds up to $\kk$-equivalence for all $n,m \in \Z$. Then for any functor $F$ on {\catr} and any integer $n$, we define $F_n(A) = F(S^n(A))$.

\begin{corol} \label{period}
Let $F$ be a functor from {\catr} to {\catabgp} that is homotopy invariant, stable, and split exact.  Then $F_*(A)$ has the structure of a graded module over the ring $K_*(\R)$.  In particular, $F_n(A) \cong F_{n+8}(A)$ for all $n \in \Z$.
\end{corol}

\begin{proof}
For all separable $A$ and $\sigma$-unital $B$, the pairing of Proposition~\ref{kkproperties} gives $\kk_*(A,B)$ the structure of a module over $\kk_*(\R, \R)$.  Taking $A = B$, we define a graded ring homomorphism $\beta$ from $K_*(\R) \cong \kk_*(\R, \R)$ to $\kk_*(A,A)$ by multiplication by $1_A \in \kk(A,A)$. Then for any $x \in F_m(A)$ and $y \in K_n(\R)$ we define 
$x \cdot y = \alpha(x \otimes \beta(y)) \in F_{n+m}(A)$.
\end{proof}

Similarly, the pairing Theorem~\ref{thm:univprop1} extends to a well-defined graded pairing 
$$\alpha \colon F_*(A) \otimes \kk_*(A,B) \rightarrow F_*(B) \; .$$

Let {\catkk} be the category whose objects are separable $R^{*}$-algebras and the set of morphisms from $A$ to $B$ is $\kk(A , B )$.  There is a canonical functor $\kk$ from {\catr} to {\catkk} that takes an object $A$ to itself and which takes a $*$-homomorphism $f \colon A \rightarrow B$ to the corresponding element $[f] \in \kk(A,B)$.    

\begin{corol}\label{thm:univprop2}
Let $F$ be a functor from {\catr} to {\catabgp} that is homotopy invariant, stable, and split exact. 
Then there exists a unique functor $\ftn{ \hat{F} }{ \catkk }{ \textbf{A} }$ such that $\hat{F} \circ \kk = F$.
\end{corol}

\begin{proof}
This statement is proved as in Section~2.8 of \cite{higson87}.
\end{proof}

\begin{propo} \label{les1}
Let $F$ be a functor from {\catr} to {\catabgp} that is homotopy invariant and half exact.  Then for any short exact sequence 
$$0 \rightarrow A \xrightarrow{f} B \xrightarrow{g} C \rightarrow 0$$ 
there is a natural boundary map $\partial \colon F(SC) \rightarrow F(A)$ that fits into a (half-infinite) long exact sequence
$$\dots \rightarrow F(SB) \xrightarrow{g_*} F(SC) \xrightarrow{\partial} 
	 F(A) \xrightarrow{f_*} F(B) \xrightarrow{g_*} F(C) \; .$$
\end{propo}

\begin{proof}
Use the mapping cone construction as in Section~21.4 of \cite{blackadarbook}.
\end{proof}

\begin{corol} \label{splitexact}
A functor $F$ from {\catr} to {\catabgp} that is homotopy invariant and half exact is also split exact.
\end{corol}

\begin{proof}
The splitting implies that $g_*$ is surjective.  Thus in the sequence of Proposition~\ref{les1}, $\partial = 0$ and $f_*$ is injective.
\end{proof}

\begin{propo}
Let $F$ be a functor from {\catr} to {\catabgp} that is homotopy invariant, stable,  and half exact.  Then for any short exact sequence 
$$0 \rightarrow A \xrightarrow{f} B \xrightarrow{g} C \rightarrow 0$$ 
there is a natural long exact sequence (with 24 distinct terms)
$$\dots \rightarrow F_{n+1}(C) \xrightarrow{\partial} F_n(A) \xrightarrow{f_*} F_n(B) \xrightarrow{g_*} F_n(C) \xrightarrow{\partial} F_{n-1}(A) \rightarrow \dots \; .$$
\end{propo}

\begin{proof}
From Corollary~\ref{splitexact} and Corollary~\ref{period}, $F$ is periodic; so Proposition~\ref{les1} gives the long exact sequence. 
\end{proof}

We say that a homotopy invariant, stable, half-exact functor $F$ from {\catr} to the category {\catabgp} of abelian groups 
\begin{enumerate}
\item[(v)] satisfies the {\it dimension axiom} if there is an isomorphism $F_*(\R) \cong K_*(\R)$ as graded modules over $K_*(\R)$
\item[(vi)] is {\it continuous} if for any direct sequence of $R^{*}$-algebras $(A_n, \phi_n)$, the natural homomorphism
$$\lim_{n \to \infty} F_*(A_n) \rightarrow F_*(\lim_{n \to \infty} (A_n)) \; $$
is an isomorphism.
\end{enumerate}

\begin{theor}
Let $F$ be a functor from {\catr} to {\catabgp} that is homotopy invariant, stable, half exact and satisfies the dimension axiom.  Then there is a natural transformation $\beta \colon K_n(A) \rightarrow F_n(A)$.  If $F$ is also continuous, then $\beta$ is an isomorphism for all $R^{*}$-algebras in the smallest class of separable $R^{*}$-algebras which contains $\R$ and is closed under $\kk$-equivalence, countable inductive limits, and the two-out-of-three rule for exact sequences.
\end{theor}

\begin{proof}
Let $z$ be a generator of $F(\R) \cong \Z$ and for $x \in K_n(A) \cong \kk(\R, S^nA)$ define 
a $K_*(\R)$-module homomorphism $\beta \colon K_*(A) \rightarrow F_*(A)$ by $\beta(x) = \alpha(z \otimes x)$.  Taking $A = \R$, Theorem~\ref{thm:univprop1} yields that $\beta(1_0) = z$ where $1_0$ is the unit of the ring $K_*(\R) = \kk_*(\R, \R)$.  Therefore, $\beta$ is an isomorphism for $A = \R$.  Then bootstrapping arguments show that $\beta$ is an isomorphism for all $R^{*}$-algebras in the class described.
\end{proof}

For any homotopy invariant, stable, split exact functor $F$ on {\catr}, define the united $F$-theory of an $R^{*}$-algebra $A$ to be
$$F\crt(A) = \{ F_*(A), F_*(\C \otimes A), F_*(T \otimes A) \} \; $$
with the module-structure given by multiplication by elements of $KK_*(X,Y)$ with $X,Y \in \{\R, \C, T\}$ via the 
pairing of Theorem~\ref{thm:univprop1}.

\begin{propo}\label{thm:crtmod}
Let $F$ be a homotopy invariant, stable, split exact functor from {\catr} to {\catabgp} and let $A$ be a separable $R^{*}$-algebra.  Then $F\crt(A)$
is a \ct-module. Moreover, if in addition $F$ is half exact, then $F\crt ( A )$ is acyclic.
\end{propo}
        
\begin{proof}
The first statement follows immediately and the second statement follows from the exact sequences
$$\dots  \rightarrow F_n(A) \xrightarrow{\eta\so} F_{n+1}(A) \xrightarrow{c} 
F_{n+1}(\C \otimes A) \xrightarrow{r \beta\su^{-1}} F_{n-1}(A) \rightarrow \dots 
$$
$$
\dots  \rightarrow F_n(A) \xrightarrow{\eta\so^2} F_{n+2}(A) \xrightarrow{\ve} 
F_{n+2}(T \otimes A) \xrightarrow{\tau \beta\st^{-1}} F_{n-1}(A) \rightarrow \dots 
$$
$$
\dots  \rightarrow F_{n+1}(\C \otimes A) \xrightarrow{\gamma} F_{n}(T \otimes A) \xrightarrow{\zeta}
F_{n}(\C \otimes A) \xrightarrow{1 - \psi\su} F_{n}(\C \otimes A) \rightarrow \dots  \; 
$$
which arise from the short exact sequences
$$0 \rightarrow S^{-1} \R \otimes A \rightarrow \R \otimes A \rightarrow \C \otimes A \rightarrow 0  
$$
$$0 \rightarrow S^{-2} \R \otimes A \rightarrow \R \otimes A \rightarrow T \otimes A \rightarrow 0  
$$
$$0 \rightarrow S \C \otimes A \rightarrow T \otimes A \rightarrow \C \otimes A \rightarrow 0  
$$
as in Sections~1.2 and 1.4 of \cite{boersema02}.  
\end{proof}

Finally, we obtain the following reduction theorem which is a formalization of a common argument used to reduce results from the complex case to the real case such as those found in \cite{BaumKaroubi04}, \cite{boersema02}, and \cite{boersema04}.

\begin{theor} \label{main}
Let $F$ and $G$ be
homotopy invariant, stable, half exact functors from {\catr} to {\catabgp} with a natural transformation $\mu_A \colon F(A) \rightarrow G(A)$.  If $\mu_A$ is an isomorphism for all $C^{*}$-algebras $A$ in {\catr}, then $\mu_A$ is an isomorphism for all $R^{*}$-algebras in {\catr}.
\end{theor}

\begin{proof}
Let $A$ be a separable $R^{*}$-algebra.  The natural transformation $\mu_A$ induces a homomorphism
$\mu_A\crt \colon F\crt(A) \rightarrow G\crt(A)$ 
of acyclic \ct-modules which is, by hypothesis, an isomorphism on the complex part.  Then the results in Section~2.3 of \cite{bousfield90} imply that $\mu_A\crt$ is an isomorphism.
\end{proof}

\begin{corol}
Any homotopy invariant, stable, half exact functor from {\catr} to {\catabgp} that vanishes on all $C^{*}$-algebras vanishes on all $R^{*}$-algebras.
\end{corol}

\section{Asymptotic Morphisms and $E$-theory} \label{E-theory}
       
Since much of the theory of asymptotic morphisms and $E$-theory carries through in the real case exactly as in the complex case, we highlight the main definitions and results without proof. The goal of this section is to show that $\kk (A, B )$ is naturally isomorphic to $E(A,B)$ for separable $R^{*}$-algebras $A$ and $B$, when $A$ is nuclear. At the end, we take advantage of Theorem \ref{main} to obtain this theorem saving us needing to check that all of the details in the complex case carry through to the real case. We note that asymptotic morphisms for $R^*$-algebras were also discussed in Section~8 of \cite{BRS}. 
       
\begin{defin}\label{d:asymorp}
Let $A$ and $B$ be $R^{*}$-algebras.  An \emph{asymptotic morphism} from $A$ to $B$ is a family $\asy{ \phi_{t} }$  (for $ t \in [1,\infty) )$ of maps from $A$ to $B$ with the following properties:
\begin{enumerate}
\item for all $a \in A$, $t \mapsto \phi_{t} ( a )$ is bounded and continuous; and

\item The set $\asy{ \phi_{t} }$ is asymptotically $*$-linear and multiplicative, i.e.\
\begin{enumerate}
\item $\displaystyle \lim_{ t \to \infty } \norm{ \phi_{t} ( \lambda a + b ) - ( \lambda \phi_{t} ( a ) + \phi_{t} ( b ) ) } = 0$;

\item $\displaystyle \lim_{ t \to \infty } \norm{ \phi_{t} ( a^{*} )  -\phi_{t} ( a )^{*} } = 0$; and 

\item $\displaystyle \lim_{ t \to \infty } \norm{ \phi_{t} ( a b ) - \phi_{t} ( a ) \phi_{t} ( b ) } = 0$
\end{enumerate}
for all $a, b \in A$ and $\lambda \in \R$.
\end{enumerate}

Two asymptotic morphisms $\ftn{ \asy{ \phi_{t} }, \asy{ \psi_{t} } }{ A }{ B  }$ are said to be \emph{equivalent} if 
\begin{align*}
\lim_{ t \to \infty } \left( \phi_{t} ( a ) - \psi_{t} ( a ) \right) = 0
\end{align*} 
for all $a \in A$; and they are said to be \emph{homotopic} if there exists an asymptotic morphism $\ftn{ \asy{ \Phi_{t} } }{ A }{ IB }$ such that for each $t \in [1, \infty )$ and $a \in A$, 
\begin{align*}
\mathrm{ev}_{0} ( \Phi_{t} (a) ) = \phi_{t} ( a )  \quad \text{and} \quad 
\mathrm{ev}_{1} ( \Phi_{t} (a) ) = \psi_{t} ( a ). 
\end{align*}
where $\mathrm{ev}_i \colon IB \rightarrow B$ are the evaluation maps for $i = 0,1$.
\end{defin}        

As in the complex case, equivalent asymptotic morphisms are always homotopic (Section~25.1.2 (g) of \cite{blackadarbook}). If $\phi$ is a $*$-homomorphism from $A$ to $B$ then $\asy{ \phi }$ denotes the corresponding constant asymptotic morphism. Also, as in Remark~25.1.4 (a) of \cite{blackadarbook}, there is a one-to-one correspondence between equivalence classes of asymptotic morphisms from $A$ to $B$; and $*$-homomorphisms from $A$ to $B_\infty = C_b([1, \infty), B)/C_0([1, \infty), B)$.

\begin{defin}
For $R^{*}$-algebras $A$ and $B$, let $[[A, B]]$ denote the set of homotopy classes of asymptotic morphisms from $A$ to $B$.  For an asymptotic morphism $\asy{ \phi_{t} }$, we will use $[ \asy{ \phi_{t} }]$ to denote the class in $[[A,B]]$ represented by $\asy{ \phi_{t} }$.  We define $E(A,B) = [[SA, \K\sur \otimes SB]]$.
\end{defin}

There is a natural isomorphism between $[[A, \K\sur \otimes B]]$ and $[[\K\sur\otimes A, \K\sur \otimes B]]$ (see Section~25.4.1 of \cite{blackadarbook}). Also as in the complex case, the set $[[A,\K\sur \otimes S B]]$ has the structure of an abelian group, defined using a chosen $*$-isomorphism $M_2(\K\sur) \cong \K\sur$. 

Let $\mathfrak{e} : 0 \to B \overset{ \iota }{ \to } E \overset{ \pi }{ \to } A \to 0$ be an exact sequence of $R^{*}$-algebras.  Suppose $B$ has a continuous approximate identity $\{ u_{t} \}_{ t \in [1, \infty ) }$ which is quasi-central in $E$.  Let $\sigma$ be a bounded continuous cross-section of $\pi$ (the existence of which is given by the (real) Bartle-Graves selection theorem).  Then the formula 
$$\phi_{t}^{ \mathfrak{e} } ( f \otimes a ) =  f( u_{t} ) \sigma (a) \qquad a \in A, f \in S\R$$
defines an
asymptotic morphism 
$\ftn{ \asy{ \phi_{t}^{ \mathfrak{e} } } }{ S A }{ B }$ and corresponding element 
$[\asy{ \phi_{t}^{ \mathfrak{e} } } ] = \ve_\mathfrak{e}$ in $[[SA, B]]$.

\begin{propo}\label{p:extasy}(See Proposition 25.5.1 in \cite{blackadarbook}.)
Let $\mathfrak{e} : 0 \to B \overset{ \iota }{ \to } E \overset{ \pi }{ \to } A \to 0$ be an exact sequence of $R^{*}$-algebras, where $B$ is separable. Then the class  $\ve_\mathfrak{e}$ is independent of the choices of continuous approximate identity $\{ u_{t} \}_{ t \in [1, \infty ) }$ and the cross-section $\sigma$.
\end{propo}

In the real case, we also have the tensor product construction for asymptotic morphisms described in Lemma II.B.$\beta$.5 of \cite{connesbook}. Given two asymptotic morphisms $\asy{\phi_t} \colon A \rightarrow C$ and $\asy{\psi_t} \colon B \rightarrow D$, there is an asymptotic morphism 
$\asy{(\phi \otimes \psi)_t} \colon A \otimes_{\rm{max}}  B \rightarrow C \otimes_{\rm{max}}  D$
which satisfies 
$\lim_{t \to \infty} \left((\phi \otimes \psi)_t(a \otimes b) - \phi_t(a) \otimes \psi_t(b) \right) = 0$
for all $a \in A$ and $b \in B$. 

As a special case, any asymptotic morphism $\asy{ \phi_{t} } \colon A \rightarrow B$ yields a suspension asymptotic morphism from $SA$ to $SB$, producing a natural element of $E(A, B)$. Similarly, the suspension construction produces a well-defined map $\Sigma \colon E(A, B) \rightarrow E(SA, SB)$, which can easily be shown to be a group homomorphism. Later in this section we show that $\Sigma$ is an isomorphism.

The associative product structure on $E$-theory described in Proposition II.B.$\beta$.4 of \cite{connesbook} also carries over to the case of $R^*$-algebras.  Given two asymptotic morphisms $\asy{\phi_t} \colon A \rightarrow B$ and $\asy{\psi_t} \colon B \rightarrow C$, there is a composition asymptotic morphism $\asy{\psi \circ \phi}_t \colon A \rightarrow C$, defined uniquely up to homotopy. In the special case that $\psi$ or $\phi$ is an actual $*$-homomorphism, then this product is a literal composition. In the general case, a reparametrization is necessary to construct an asymptotic morphism from the composition (see Section~25.3 of \cite{blackadarbook}). The resulting product induces a natural homomorphism $E(A, B) \otimes E(B, C) \rightarrow E(A, C)$. 

\begin{theor} \label{Eproperties}
$E(A,B)$ is a bivariant functor from separable $R^{*}$-algebras to abelian groups. In both arguments, it is homotopy invariant, stable, half exact, and has a degree 8 periodicity isomorphism. 
\end{theor}

\begin{proof}
The homotopy invariance is immediate and the stability follows from Proposition~\ref{stability}.  By Proposition~\ref{p:extasy}, any extension
$$\mathfrak{e} : 0 \rightarrow J \rightarrow A \xrightarrow{q} B \rightarrow 0$$
gives rise to a well-defined asymptotic morphism $\asy{ \phi_{t}^{ \mathfrak{e} } }$ from $S B$ to $J$.  Then the proofs leading up to and including Corollary~25.5.7 of \cite{blackadarbook} carry over to the real case to show that the functor $E(A,\cdot)$ is a split exact functor for fixed separable $A$.  Then by Theorem~\ref{thm:univprop1} (of the present paper) there is a bilinear pairing $E(A,B) \otimes \kk(B,C) \rightarrow E(A,C)$. Since this map is associative, multiplication by the Bott element in $\kk(\R, S^8 \R)$ induces a periodicity isomorphism in the second argument of $E(\cdot, \cdot)$. Similarly, $E( \cdot, B)$ is also split exact, so by the comments following Theorem~\ref{thm:univprop1} there is a natural pairing, $\kk(A, B) \otimes E(B, C) \rightarrow E(A,C)$ proving periodicity in the first argument.

We postpone the proof of half-exactness until after the following lemma.
\end{proof}

For any elements $x \in E(A, B)$ and $z \in \kk(B, C)$, the pairing described in the proof above gives an element in $E(A, C)$ which we will denote by $x \otimes_\alpha z$. Taking $1_A \in E(A, A)$ we obtain a homomorphism $\ve \colon \kk(A, B) \rightarrow E(A, B)$. For $x \in E(A, B)$ and $y \in E(B, C)$, we let $x \otimes_E y$ denote the product in $E(A, C)$.  Similarly for $z \in \kk(A, B)$ and $w \in \kk(B, C)$, we let $z \otimes_{KK} w$ denote the product in $\kk(A, C)$. For $x \in \kk(A, B)$ and $y \in \kk(B, C)$ it can easily be shown that 
$\ve(x \otimes_{KK} y) = \ve(x) \otimes_E \ve (y)$.

\begin{lemma} \label{suspension}
For any $R^*$-algebras, $\Sigma \colon E(A, B) \rightarrow E(SA, SB)$ is an isomorphism.
\end{lemma}

\begin{proof}
Let $\alpha \in E(\R, S^8 \R)$ and $\beta \in E(S^8 \R, \R)$ be Bott elements arising from the corresponding Bott elements in $\kk$-theory via $\ve$. Since $\ve$ is multiplicative, it follows that these elements satisfy $\alpha \otimes_E \beta = 1_\R$ and $\beta \otimes \alpha = 1_{S^8 \R}$.  It follows that the map $z \mapsto (\alpha \otimes 1_A) \otimes_E z \otimes (\beta \otimes 1_B)$ is an isomorphism $\Theta$ from $E(S^8 A, S^8 B)$ to $E(A, B)$.

It can easily be shown that $\Theta \circ \Sigma^8 = \id_{E(A,B)}$ and that $\Sigma^8 \circ \Theta = \id_{E(S^8A, S^8 B)}$. Hence $\Sigma \colon E(S^n A, S^n B) \rightarrow E(S^{n+1} A, S^{n+1} B)$ is an isomorphism if $n \geq 7$. Finally, $\Sigma \colon E(A, B) \rightarrow E(SA, SB)$ can be shown to be an isomorphism using the diagram
\begin{equation*}
\xymatrix{
E(A, B)  \ar[rr]^\Sigma
&& E(SA, SB) \\
E(S^8 A, S^8 B)  \ar[rr]^\Sigma \ar[u]^\Theta
&& E(S^9 A, S^9 B) \ar[u]^\Theta  \; .
}
\end{equation*}
The vertical maps and the lower map are all known to be isomorphisms.
This diagram commutes modulo a homomorphism induced by the rearrangement of the order of the suspension factors of $S^9 A$ and $S^9 B$. Since the rearrangement of the factors corresponds to an even permutation, it is homotopic to the identity and induces an identity homomorphism on $E$-theory.  It follows that the $\Sigma \colon E(A, B) \rightarrow E(S A, S B)$ is an isomorphism.
\end{proof}

\begin{proof}[Completion of proof of Theorem~\ref{Eproperties}.]

Finally, to prove half-exactness let $0 \to J \overset{ \iota }{ \to } A \overset{ q }{ \to } B \to 0$ be an extension of separable $R^{*}$-algebras and $D$ be an $R^{*}$-algebra. Suppose $h$ is an asymptotic morphism from $SD \otimes \K_{ \R }$ to $SA \otimes \K_{ \R }$ such that $[ q \circ h ] = [ 0 ]$. Lemma~25.5.12 of \cite{blackadarbook} and its proof carry over to the real case, so there exists an asymptotic morphism $k$ from $S^2 D \otimes \K_{ \R }$ to $S^2 J \otimes \K_{ \R }$ such that $[ S \iota \circ k ] = [  S h ]$.  
In the commutative diagram
\begin{equation*}
\xymatrix{
E( D , J ) \ar[r]^{ \iota_{*} } \ar[d]^{\Sigma} & E( D , A ) \ar[r]^{ q_{*} } \ar[d]^{\Sigma} & E( D , B ) \ar[d]^{\Sigma}  \\
E( S D , S J ) \ar[r]^{ \iota_{*} }  & E( S D , S A ) \ar[r]^{ q_{*} } & E ( S D , S B ) \\
}
\end{equation*}
the vertical maps are isomorphisms, so there exists an asymptotic morphism $g$ from $SD \otimes \K_{ \R }$ to $SJ \otimes \K_{ \R }$ such that $[  \iota \circ g ] = [ h ]$ in $E( D , A )$.  

Half-exactness in the first argument is proved in a similar way.
\end{proof}

\begin{theor}\label{t:kkthyethyiso}
Let $A$ be a separable, nuclear $R^{*}$-algebra and let $B$ be a separable $R^{*}$-algebra. Then the homomorphism
$\ve \colon \kk( A , B ) \rightarrow E ( A , B )$ is an isomorphism.
\end{theor}   

\begin{proof}
By Theorem~\ref{main}, it suffices to show that $\kk(A,B) \rightarrow E(A,B)$ is an isomorphism when $B$ is a $C^{*}$-algebra. In the diagram
\begin{equation*}
\xymatrix{
\kk^{\C}(A\suc, B) \ar[r]^{~} \ar[d]  &
E^{\C}(A\suc, B) \ar[d] \\
\kk(A, B) \ar[r] &
E(A, B)   }
\end{equation*}
we use $\kk^{\C}(-,-)$ and $E^{\C}(-,-)$ to denote the versions of these functors on $C^{*}$-algebras (for example $E^{\C}(-,-)$ consists of homotopy classes of asymptotic morphisms that are asymptotically linear over $\C$). Since $A\suc$ is nuclear, the top horizontal homomorphism is an isomorphism by Theorem~25.6.3 of \cite{blackadarbook}. The left vertical homomorphism is an isomorphism by Lemma~4.3 of \cite{boersema04}. The right vertical homomorphism is defined by restriction---a complex asymptotic morphism defined on $S A\suc \otimes \K$ restricts to a real asymptotic morphism defined on $S A \otimes \K\sur$---and it is an isomorphism since every real asymptotic morphism on $S A \otimes \K\sur$ can be extended uniquely to a complex asymptotic morphism on 
$S A\suc \otimes \K$. The square commutes by Theorem~3.7 of \cite{higson87}, since the two directions around the square give natural transformations $\kk^{\C}(A\suc, -) \rightarrow E(A, -)$ of functors defined on separable $C^{*}$-algebras, each of which sends $1_A \in \kk^{\C}(A\suc, A\suc)$ to the asymptotic morphism represented by the inclusion of $A$ into $A\suc$. Therefore, the bottom row is an isomorphism as desired.
\end{proof}

Finally, we mention that with a similar application of Theorem~\ref{main}, we can easily obtain the following real analog of Theorem 5.8 of \cite{ManuilovThomsen04}, showing that $E$-theory is a special case of $\kk$-theory. 
Since $\mathcal{M}(B \otimes \K\sur)$ is $\kk$-trivial, it follows that $E(A, \mathcal{M}(B \otimes \K\sur)) = 0$.
Then we use the long exact sequence arising from 
$$0 \rightarrow B \otimes \K\sur \rightarrow \mathcal{M}(B \otimes \K\sur) \rightarrow \mathcal{Q}(B \otimes \K\sur) \rightarrow 0$$
to get an isomorphism
$E_0(A, \mathcal{Q}(B \otimes \K\sur)) \cong E_{-1}(A, B \otimes \K\sur)$. Combining this with stability and with Lemma~\ref{suspension}, there is an isomorphism
$$\gamma \colon E( S^{-1}A , \mathcal{Q}(B \otimes \K\sur) ) \rightarrow E ( A , B )  $$
for $R^*$-algebras $A$ and $B$.

\begin{theor}
Let $A$ and $B$ be a separable $R^{*}$-algebras. Then there is an isomorphism  \\
$$\ve' \colon \kk( S^{-1}A , \mathcal{Q}(B \otimes \K\sur) ) \rightarrow E ( A , B ) \; .$$
\end{theor}

\begin{proof}
Let $\ve' = \gamma \circ \ve$ where $\ve$ is the isomorphism of Theorem~\ref{t:kkthyethyiso}. We know that $\ve'$ is an isomorphism in the complex case by Theorem 5.8 of \cite{ManuilovThomsen04}. So it suffices by Theorem~\ref{main} to show that $\ve'$ fits in a commutative square in the same way that $\ve$ does in the proof of Theorem~\ref{t:kkthyethyiso}. The square in question can be factored into two squares as follows:
\begin{equation*}
\xymatrix{
\kk^{\C}(S^{-1} A\suc, \mathcal{Q}(B \otimes \K\sur)) \ar[r]^{~} \ar[d]  &
E^{\C}(S^{-1} A\suc, \mathcal{Q}(B \otimes \K\sur)) \ar[r] \ar[d] &
E^{\C}(A\suc, B) \ar[d] \\
\kk(S^{-1} A, \mathcal{Q}(B \otimes \K\sur)) \ar[r] &
E(S^{-1} A, \mathcal{Q}(B \otimes \K\sur)) \ar[r] &
E(A, B)   }
\end{equation*}
The first square is just a specialization of the square in Theorem~\ref{t:kkthyethyiso}. The second square commutes since the horizontal maps are homomorphisms that arise from stabilization, from long exact sequences, and from suspensions; all of which commute with the vertical restriction map.
\end{proof}

\section{Application:  asymptotic morphisms on spheres}

The goal of this section is to determine when there exist asymptotic morphisms from suspensions of $\R$ to $\K_{\R}$ and to $\K_{\R} \otimes \Ham$, that are detected by $K$-theory. For this we will use the results of the previous sections, the universal coefficient theorem for real $C^{*}$-algebras, and a united $K$-theory analysis. The main theorem of this section is the following.

\begin{theor}\label{t:asynonvanishing}
Let $d \in \N$.
There exists an asymptotic morphism $ \asy{ \phi_{t}  }$ that induces a non-trivial homomorphism on $K$-theory of the form
\begin{enumerate}
\item  ${ S^d \R } \rightarrow { \K_{\R}   }$ if and only if $d \equiv 0, 4, 6, 7 \pmod 8$, 
\item  ${ S^d \R } \rightarrow { \K_{\R} \otimes \Ham }$ if and only if $d \equiv 0, 2, 3, 4\pmod 8$ 
\end{enumerate}
\end{theor}

The correspondence $\asy{ \phi_{t} }  \mapsto \asy{ \phi_{t} }_* $ gives a map
$$[[A, B]] \rightarrow \hoom\scrt(K\crt(A), K\crt(B))$$
that respects compositions the expected way and, when $B$ is stable, is a homomorphism of semi-groups. Therefore, there is a group homomorphism
$\gamma'$ defined on $E(A, B)$ that can be shown using Theorem~\ref{thm:univprop1} to commute in the following diagram.
\[
\xymatrix{
KK(A,B) \ar[rd]_\ve \ar[rr]^\gamma && \hoom\scrt(K\crt(A), K\crt(B)) \\
& E(A,B) \ar[ru]_{\gamma'}
}
\]

Now in that case that the first argument is a suspension algebra $SA$, we can further factor $\ve$ as in the following diagram
\[
\xymatrix{
&KK(SA,B) \ar[rd]^\ve \ar[ld]_{\ve'} \\
[[SA, \K\sur \otimes B]] \ar[rr]^\Sigma && E(SA,B) = [[S^2 A, \K\sur \otimes SB]]  
}
\]
where the map $\ve'$ is described as follows. An element of $KK(SA, B)$ is associated with an extension 
$\mathfrak{e} : 0 \to B \overset{ \iota }{ \to } E \overset{ \pi }{ \to } A \to 0$ using the isomorphism 
$KK(SA, B) \cong \ext(A, B)^{-1}$ of Kasparov (Theorem~7.1 of \cite{kasparov80}).  Then the construction of Theorem~\ref{p:extasy} produces an asymptotic morphism $\ve_\frak{e}$
from $SA$ to $B$, giving an element of $[[SA, \K\sur \otimes B]]$. Again with a little work, Theorem~\ref{thm:univprop1} can be used to show that the diagram commutes (also see the comment following Corollary~25.5.8 in \cite{blackadarbook}).

We are not claiming to have shown that $\ve'$ is an isomorphism although that is the case in the complex setting (see Corollary 5.3 of \cite{DadarlatLoring94}). We hope to address unsuspended $E$-theory for $R^*$-algebras more thoroughly in future work. For our current purposes it is enough to know that $\gamma$ factors through $\ve'$.

\begin{theor}\label{t:CRTnonvanishing}
Let $d \in \N$. There exists a non-trivial \ct-homomorphism of degree 0 of the form
\begin{enumerate}
\item $K\crt(S^d \R) \rightarrow K\crt(\R)$ if and only if $d \equiv 0, 4, 6, 7 \pmod 8$
\item  $K\crt(S^d \R) \rightarrow K\crt(\Ham)$ if and only if $d \equiv 0, 2, 3, 4\pmod 8$.
\end{enumerate}
\end{theor}

\begin{proof}
The $\ct$-module $K\crt(S^d \R)$ is a free \ct-module with a single generator in the real part in degree $-d$. Hence there exists a non-trivial \ct-module homomorphism $K\crt(S^d \R) \rightarrow M$ if and only if $M\so^{-d} \neq 0$. Now $K_*(\R)$ is non-zero in and only in degrees $0, 1, 2, 4 \pmod 8$; and $K_*(\Ham)$ is non-zero in and only in degrees $0, 4, 5, 6 \pmod 8$. Thus parts (1) and (2) follow as well as the converse statements.
\end{proof}

We are now ready to prove Theorem~\ref{t:asynonvanishing}.

\begin{proof}[Proof of Theorem \ref{t:asynonvanishing}]  
For $d \geq 1$, where $d$ is not in one of the required congruence classes, Theorem~\ref{t:CRTnonvanishing} implies that no asymptotic morphism can exist that induces a non-trivial homomorphism on $K$-theory.

Now suppose that $d \geq 1$ and suppose that $A =  S^d \R$ and $B = \R$  or $B = \Ham$ are algebras such that the form $A \rightarrow B$ matches the conditions of the statement of Theorem~\ref{t:asynonvanishing}. By Theorem~\ref{t:CRTnonvanishing}, there is a non-zero homomorphism $K\crt(A) \rightarrow K\crt(B)$ and by the Universal Coefficient Theorem for real $C^{*}$-algebras (Theorem~1.1 of \cite{boersema04}), this \ct-module homomorphism is induced by a nonzero element $\xi \in \kk(A, B)$. Then $\ve'(\xi)$ is a class in $[[A, \K\sur \otimes B]]$ and when we choose a representative we obtain an asymptotic morphism that induces a non-trivial map on united $K$-theory. \end{proof}

\section{Almost commuting matrices}

We now use this machinery to find novel examples of almost
commuting real symmetric matrices. Our approach is to use commutative
$R^{*}$-algebras and create asymptotic morphisms out of these. On
the one hand these carry $K$-theory data that can distinguish them
from actual $*$-homomorphisms. On the other, in the image the relations
that make the $R^{*}$-algebra commutative get turned into the property
of being almost commuting.

Let $\F$ denote either $\R$ or $\C$ or $\Ham$ and let $d \in \N$. We generalize a question of Halmos \cite{DavidsonEssentially2010}
to ask about $d$ almost commuting self-adjoint matrices over $\F$. 

\begin{probl}
\label{HalmosProblem} For all $\ve>0$, does there exist $\delta>0$
so that, for all $n$, given $d$ self-adjoint contractions $H_{r}$
in $\mathbf{M}_{n}(\F)$ such that
\[
\left\Vert \left[H_{r},H_{s}\right]\right\Vert \leq\delta,
\]
there exist $d$ self-adjoint contractions $K_{r}$ with
$\left\Vert K_{r}-H_{r}\right\Vert \leq\ve$
and 
\[
\left[K_{r},K_{s}\right]=0?
\]
\end{probl}

In the complex case, Lin \cite{LinAlmostCommutingHermitian} showed that in the 
anwer is ``yes'' for $d=2$ while it was known much earlier
\cite{VoiculescuHeisenberg} that the answer is ``no'' for $d = 3$.
For the quaternionic case, the result is the same: yes for $d=2$
\cite{LorSorensenTorus} and no for $d=3$ \cite{HastLorTheoryPractice}.

These leaves the real case, arguably the most important case. We know the
answer is ``yes'' for $d=2$ (\cite{LorSorensenTorus}). We will show that the answer is ``no'' for $d = 5$, leaving open the cases $d = 3,4$.
The proof techniques used for a negative result for $d=3$ in the complex and quaternionic cases rely on the fact
that $K_{-2}(\Ham) \neq 0$ and
$K_{-2}(\C) \neq 0$ and so will not
work for $\F=\R$ since $K_{-2}(\R) = 0$.
However, since $K_{-4}(\R)$ is nontrivial, we will see that these methods will apply for $d = 5$.

We start by connecting this problem to a problem couched in the theory of $R^*$-algebras. For any sequence $B_n$ of $R^*$-algebras, let $\pi$ be the quotient map from the product
$\prod_{n=1}^\infty B_n$ to its quotient by the sum $\left.\prod_{n=1}^{\infty} B_n \right/\bigoplus_{n=1}^{\infty} B_n $.

\begin{probl}
\label{HalmosProblem2} Does every $*$-homomorphism of the form
$$\psi \colon S^{d-1} \R \rightarrow \left.\prod_{n=1}^{\infty} \mathbf{M}_{m(n)}(\F) \right/\bigoplus_{n=1}^{\infty} \mathbf{M}_{m(n)}(\F) $$
where $\{m(n)\}_{n=1}^\infty$ is a sequence of integers,
lift to a $*$-homomorphism
$$\tilde{\psi} \colon S^{d-1} \R \rightarrow \prod_{n=1}^\infty \mathbf{M}_{m(n)}(\F)$$
such that $\psi = \pi  \circ \widetilde{\psi}$?
\end{probl}

\begin{thm} \label{2Problems}
For a fixed positive integer $d$ and division algebra $\F$, if the answer to Problem~\ref{HalmosProblem} is ``yes'' then the answer to Problem~\ref{HalmosProblem2} is also ``yes''. 
\end{thm}

We prove Theorem~\ref{2Problems} below following this basic lemma stating the universal properties of $C(S^{d-1}, \R)$. This lemma can be proven using the techniques of Chapter~3 of \cite{loringbook}.

\begin{lemma} \label{classifyingspaces}
If $h_1, \dots, h_d$ are commuting self-adjoint contractions in a $R^*$-algebra $A$ that satisfy $\sum_{j = 1}^d h_j^{2} = 1$, then there is a unique $*$-homomorphism
$\psi \colon C(S^{d-1}, \R) \rightarrow A$ sending the $j$th coordinate function $f_j$ to $h_j$.
\end{lemma}

\begin{proof}[Proof of Theorem~\ref{2Problems}]
Suppose that the answer to Problem~\ref{HalmosProblem} is ``yes'' for some $d$ and $\F$, and let 
$$\psi \colon C(S^{d-1}, \F) \rightarrow 
	\left.\prod_{n=1}^{\infty} \mathbf{M}_{m(n)}(\F) \right/\bigoplus_{n=1}^{\infty} \mathbf{M}_{m(n)}(\F) $$
be a $*$-homomorphism. Then taking the images of the coordinate functions in $C(S^{d-1}, \R)$ and lifting them to representatives in 
$\left.\prod_{n=1}^{\infty} \mathbf{M}_{m(n)}(\F) \right. $
we find that there exist sequences of matrices $H_{in} \in \mathbf{M}_{m(n)} (\R)$ ($i \in \{1, \dots, d\}$, $n \in \N$) that are asymptotically (as $n \to \infty$) self-adjoint contractions,  that satisfy $\sum_{i = 1}^d H_{in}^2 = 1$ asymptotically, and such that  $H_{in}$ and $H_{jn}$ asymptotically commute for each $i,j \in \{1, \dots, d\}$. We may assume that each $H_{in}$ is exactly self-adjoint by replacing $H_{in}$ by $\tfrac{1}{2} \left( H_{in} + H_{in}^* \right)$.

By our hypothesis, there exists sequences of self-conjugate contractions $K_{in} \in \mathbf{M}_{m(n)} ( \F)$ that exactly commute and satisfy $\lim_{n \to \infty} (H_{in} - K_{in}) = 0$ for each $i$. Furthermore, by normalizing, we may assume that $\sum_{i = 1}^d K_{in}^2 = 1$ holds for each $n$.

Then by Lemma~\ref{classifyingspaces}, there exist $*$-homomorphisms $\psi'_n \colon C(S^{d-1}, \F) \rightarrow \mathbf{M}_{m(n)} (\F)$ that map the $d$ coordinate functions to $K_{in}$ which together form the desired lift of $\psi$.
\end{proof}

The rest of the section is devoted to showing that the answer to Problem~\ref{HalmosProblem2} is ``no'' when $d = 5$, using a $K$-theoretic obstruction. 

\begin{lemma} \label{nonzero}
Given an asymptotic homomorphism $\langle \phi \rangle$ from $A$ to $\K_\F$ such that $\langle \phi \rangle_*$ is nonzero, then there exists a homomorphism $\psi \colon A \rightarrow \prod_{n=1}^\infty \K_\F/ \bigoplus_{n=1}^\infty \K_\F$ such that $\psi_*$ is nonzero.
\end{lemma}

\begin{proof}Given an asymptotic homomorphism $\langle \phi \rangle$ from $A$ to $B = \K_\F$, there is a corresponding homomorphism from $A$ to $B_\infty = C_b([1, \infty))/C_0([1, \infty))$, as discussed in Section~\ref{E-theory}. By evaluating at the positive integers we also obtain a discrete version, that is a map $\psi$ from $A$ to 
$B^{\rm d}_\infty = \prod_{n=1}^{\infty}B / \bigoplus_{n=1}^\infty B$. Note that as in the complex case (see Section~3.2 of \cite{DadarlatEilers02}), we have $K\crt(B^{\rm d}_\infty) \cong \prod_{n=1}^\infty K\crt(B)/ \bigoplus_{n = 1}^\infty K\crt(B)$. But as each evaluation map is homotopic to each other evaluation map, the map 
$$\psi_* \colon K\crt(A) \rightarrow \prod_{n=1}^\infty K\crt(B)/ \bigoplus_{n = 1}^\infty K\crt(B) $$
is given by
$$x \mapsto \pi \circ \Delta \circ \langle \phi \rangle_*(x)$$
where $\Delta$ is the diagonal map $\pi$ is the quotient map.
In particular, if $\langle \phi \rangle_*$ is non-zero, so is $\psi_*$.
\end{proof}

\begin{lemma} \label{fgsubalgebra}
Suppose $B$ is a separable $R^{*}$-algebra, and $A$ is a finitely generated $R^{*}$-subalgebra of
\[
A \subseteq
\left.\prod_{n=1}^{\infty} (B\otimes \K_\R )\right/\bigoplus_{n=1}^{\infty} (B\otimes \K_\R)  \; .
\]
Then there is
a sequence $m(1)<m(2)<$ on natural numbers so that 
\[
A \subseteq
\left.\prod_{n=1}^{\infty}\mathbf{M}_{m(n)}(B)\right/\bigoplus_{n=1}^{\infty}\mathbf{M}_{m(n)}(B).
\]
\end{lemma}

\begin{proof}
Given a single element $a \in A$, write $a = [(a_{1},a_{2},\dots)]$ where $a_i \in B \otimes \K\sur$.
We can choose an increasing sequence $p_{1},p_{2}, \dots$ of standard projections
in $1\otimes\K_{\R}$ (with $1$ in $\tilde{B}$ if
needed) so that 
\[
\left\Vert p_{n}a_{n}p_{n}-a_{n}\right\Vert \leq\frac{1}{n}
\]
 and so 
$[(a_{1}, a_{2},\dots)] = [(p_{1}a_{1}p_{1},p_{2}a_{2}p_{2},\dots)]$.
More generally, for a finite set of elements in $A$, we can use a single sequence of
projections as above to show that 
$$A \subseteq \left.\ \prod_{n=1}^\infty p_n (B \otimes \K\sur) p_n  \right/  \bigoplus_{n=1}^\infty p_n(B \otimes \K\sur)p_n \; .$$
\end{proof}

We are ready to prove our main theorem.

\begin{thm}
Suppose that $d \in \N$ and $\F \in \{\R, \C,  \Ham\}$ satisfy the hypotheses of Theorem~\ref{t:asynonvanishing}. Then there is a sequence of integers $m(1),m(2),\dots$
and a unital $*$-homomorphism
\[
\phi:C(S^{d},\R)\rightarrow\left. \prod_{n=1}^{\infty}\mathbf{M}_{m(n)}(\F)\right/\bigoplus_{n=1}^{\infty}\mathbf{M}_{m(n)}(\F)
\]
 that cannot be lifted to a unital $*$-homomorphism to
$\prod_{n=1}^{\infty} \mathbf{M}_{m(n)}(\F)$.
\end{thm}

\begin{proof}
Let $d$ be as above. By Theorem~\ref{t:asynonvanishing}, there exists an asymptotic morphism 
\[
\left\langle \phi_{t}\right\rangle :S^d \R \rightarrow \K_\F
\]
that induces a non-zero map on $K$-theory.
Then by Lemma~\ref{nonzero}, 
we obtain a $*$-homomorphism of the form
$$
\phi' \colon S^d \R \rightarrow
\left.\prod_{n=1}^{\infty} \K_\F \right/\bigoplus_{n=1}^{\infty} \K_\F \; 
$$
that is non-zero on $K$-theory. By Lemma~\ref{fgsubalgebra}, this $*$-homomorphism factors through a $*$-homomorphism of the form
$$
\phi: S^d \R \rightarrow
\left.\prod_{n=1}^{\infty}\mathbf{M}_{m(n)}(\F)\right/\bigoplus_{n=1}^{\infty}\mathbf{M}_{m(n)}(\F) \; .
$$
that must also be nonzero on united $K$-theory.

Now if $\phi$ could be lifted to a $*$-homomorphism
with values in $\prod_{n=1}^{\infty} \mathbf{M}_{m(n)}(\F)$
then such a lift would have to be non-zero on $K$-theory. However, as any homomorphism from $S^d$ to $\K_F$ vanishes on $K$-theory, no such lift of $\phi$ exists.

Finally, extend unitally to form a $*$-homomorphism
$$
\phi: C(S^d, \R) \rightarrow
\left.\prod_{n=1}^{\infty}\mathbf{M}_{m(n)}(\F)\right/\bigoplus_{n=1}^{\infty}\mathbf{M}_{m(n)}(\F) \; 
$$
that similarly cannot be lifted.
\end{proof}

\begin{corol}
For $d = 5$, $\F = \R$ and for $d = 3$, $\F = \Ham$; the answer to Problem~\ref{HalmosProblem2}, and hence also to Problem~\ref{HalmosProblem}, is ``no''.
\end{corol}

The result above for $d = 3$, $\F = \Ham$, replicates results from \cite{LorSorensenTorus}. 

\section{Pictures of K-theory}

It is standard practice to represent $K_0$ classes by projections and $K_1$ classes by
unitaries. Since \cite{HastLorTheoryPractice}, a picture has been developing in which 
all ten of the real and complex $K$-theory groups of an $R^*$-algebra can be represented concretely
in terms of homotopy classes of unitary elements with certain symmetries. We finish this 
paper showing how our methods partially extend this picture to $K_{-1}(A)$ and $K_{3}(A)$, 
allowing us to represent any such element with a specific class of unitaries.
We expect that a complete study of unsuspended $E$-theory in the case of $R^*$-algebras
can be used to complete this picture, including a concrete description of all of the interrelating 
natural transformations and the boundary maps.

The table below summarizes this picture and extends Tables~7 and 8 of \cite{HastLorTheoryPractice}, showing 
the symmetries that are used to represent each $K$-group. Here $A$ is an $R^*$-algebra 
and $\tau$ is the associated anti-automorphism of the complexification $A\suc$. Thus, the first 
line indicates that $KU_0(A)$ is isomorphic to the group of homotopy classes of self-adjoint unitaries 
in $M_{\infty}(A\suc)$ while $KO_0(A)$ is isomorphic to the group of homotopy classes of 
self-adjoint unitaries satisfying $\tau(u) = u^*$ 
(a priori, $u$ is in $A\suc$, but the second condition restricts it to $A$).

\begin{center} 
\begin{tabular}{|c|c|}
\hline 
K-group & unitary classes  \\ \hline \hline
$KU_0(A)$ & $u = u^*$ \\ \hline
$KU_1(A)$ &  -- \\ \hline  \hline
$KO_{-1}(A)$ & $u^\tau = u $ \\ \hline
$KO_0(A)$ &  $u = u^*$, $u^{\tau} = u^* $  \\ \hline
$KO_1(A)$ &  $u^{\tau} = u^* $ \\ \hline 
$KO_2(A)$ &  $u = u^*$, $u^\tau = -u$ \\ \hline 
$KO_{3}(A)$ & $u^{\tau\otimes \sharp} = u $ \\ \hline
$KO_4(A)$ &  $u = u^*$, $u^{\tau \otimes \sharp} = u^*$  \\ \hline
$KO_5(A)$ &  $u^{\tau \otimes \sharp} = u^*$ \\ \hline 
$KO_6(A)$ &  $u = u^*$, $u^{\tau \otimes \sharp} = -u$ \\ \hline \hline
\end{tabular}
\end{center}

In this table, the line with $KO_{-1}(A)$ comes from Theorem~\ref{KO-unitary} below. 
The line for $KO_{3}(A)$ then arises via the isomorphism 
$KO_{n+4}(A) \cong KO_{n}(\Ham \otimes \sur A)$. The
the antiautomorphism of $A\suc \otimes\suc  M_2(\C)$ 
associated to $\Ham \otimes\sur A$ is $\sharp \otimes \tau \otimes$
where
$$\sharp \colon \sm{a}{b}{c}{d} \rightarrow \sm{d}{-b}{-c}{a} \; .$$
All the other lines of the 
table are discussed in \cite{HastLorTheoryPractice}.

Let $A$ be an $R^*$-algebra and let $A\suc$ be the complexification with anti-automorphism $\tau$.
Let $G(A)$ be the group of homotopy classes of unitaries $u \in M_{\infty}(\widetilde{A\suc})$ that 
satisfy $u^\tau = u$. The associated antiautomorphism 
on $M_n(A\suc) \cong  M_n(\C) \otimes A\suc$ 
is $\rm{tr} \otimes \tau$
and we have
$M_{n}(\widetilde{A})$ mapping into $M_{n+1}(\widetilde{A})$ 
by $u \mapsto \sm{u}{0}{0}{1}$. 

\begin{lemma} \label{group}
For an $R^*$-algebra $A$, $G(A)$ is an abelian group.
\end{lemma}

\begin{proof}
Let $R(t) = \sm{\cos (\pi t/2)}{-\sin (\pi t/2)}{\sin (\pi t/2)}{\cos (\pi t/2)}$ 
be the rotation matrix for 
$t \in [0, 1]$.  Since $R^{\rm{tr}} = R^*$, for any unitaries $u_1, u_2 \in A$ satisfying $u_i^\tau = u_i$, 
we have that
$U(t) = R(t) \sm{u_1}{0}{0}{u_2} R(t)^*$ is a path of unitaries satisfying $U(t)^{\tau} = U(t)$ from
$\sm{u_1}{0}{0}{u_2}$ to $\sm{u_2}{0}{0}{u_1}$
showing that $G(A)$ is a commutative semigroup.

To show that there are inverses in $G(A)$, it suffices to show that 
$\sm{z}{0}{0}{\overline{z}}$ is homotopic to $1_2$ in $C(M_2(S^1))$, through homotopies satisfying
$u^\tau = u$, taking the trivial involution on $C(S^1, \C)$. Indeed,
$z \in C(S^1, \C)$ is the universal unitary satisfying $u^\tau = u$.
It is easy to see that $\sm{z}{0}{0}{\overline{z}}$ is homotopic to
$$\begin{cases}
\sm{z^2}{0}{0}{1} & \im{z} \geq 0 \\
\sm{1}{0}{0}{\overline{z}^2} & \im{z} < 0 \; .
\end{cases}
$$
Using a variation of the rotation argument of the first paragraph, this is homotopic to
$$\begin{cases}
\sm{z^2}{0}{0}{1} & \im{z} \geq 0 \\
\sm{\overline{z}^2}{0}{0}{1} & \im{z} < 0 \; .
\end{cases}
$$
which is clearly homotopic to $\sm{1}{0}{0}{1}$.
\end{proof}

\begin{thm} \label{KO-unitary}
There is a natural group homomorphism $\phi_A \colon K_{-1}(A) \rightarrow G(A)$ with a left
inverse.
\end{thm}

This theorem will be proven easily from the machinery already established 
in this paper. 
It follows that any unitary $u$ satisfying $u^\tau = u$ determines
a class in $K_{-1}(A)$; that any class in $K_{-1}(A)$ arises from such a unitary in this way; and
that two distinct $K_{-1}$ classes must arise from distinct
homotopy classes of unitaries. That $\phi_A$ is an isomorphism will be left to further work.

\begin{proof}[Proof of Theorem~\ref{KO-unitary}] 
We use the homomorphism
$$KK(S\R, A) \xrightarrow{\ve'} [[S\R, \K\sur \otimes \sur A]]$$
discussed in Section~5 which was shown there to have a left inverse. As there is a natural
isomorphism $K_{-1}(A) \cong KK(S\R, A)$, it only remains to establish an isomorphism between
$[[S\R, K\sur \otimes \sur A]]$ and $G(A)$.

From Section~3 of \cite{LorSor} we know that $S\R$ is semiprojective, which gives us the first 
isomorphisms in the following chain.
\begin{align*}
[[S\R, K\sur \otimes \sur A]] &\cong [S\R, \K\sur \otimes \sur A]   \\
&\cong [\widetilde{S\R}, \widetilde{\K\sur \otimes \sur A}]^+  \\
& \cong [(\widetilde{S\C}, \id), (\widetilde{\K \otimes A\suc}, \id \otimes \tau)]^+ \\
& \cong G(A)
\end{align*}
The third isomorphism comes from the catagorical equivalence between (unital) $R^*$-algebras 
and (unital) $C^*$-algebras with antiautomorphism.
The fourth isomorphism arises from the fact that
$\widetilde{S\C} = C(S^1, \C)$ is the 
universal $C^*$-algebra generated by
a unitary $u$ satisfying $u^\tau = u$ as in the proof of Lemma~\ref{group}.
\end{proof}

\providecommand{\bysame}{\leavevmode\hbox to3em{\hrulefill}\thinspace}
\providecommand{\MR}{\relax\ifhmode\unskip\space\fi MR }
\providecommand{\MRhref}[2]{
  \href{http://www.ams.org/mathscinet-getitem?mr=#1}{#2}
      }
\providecommand{\href}[2]{#2}

\end{document}